\documentclass{amsart}
\usepackage{graphicx}
\usepackage{amsmath,amsthm,amsfonts,amscd,amssymb,mathrsfs,cite, hyperref}
\usepackage[all]{xypic}
\usepackage[enableskew]{youngtab} 

\newtheorem{theorem}{Theorem}[section]
\newtheorem{conjecture}[theorem]{Conjecture}
\newtheorem{lemma}[theorem]{Lemma}

\newtheorem{prop}[theorem]{Proposition}

\theoremstyle{definition}

\theoremstyle{remark}
\newtheorem{remark}[theorem]{Remark}
\newcommand{\CC}{\mathbb C}
\newcommand{\PP}{\mathbb P}
\newcommand{\V}{\mathcal{V}}

\newcommand{\tdt}{\times\dots\times}
\newcommand{\otdot}{\otimes\dots\otimes}

\newcommand{\G}{\left( SL(2)^{\times n}\right) \ltimes \mathfrak{S}_{n}}
\newcommand{\tS}{\tau \left(Seg\left(\PP V^{*}_{1}\tdt \PP V^{*}_{n} \right)\right)}
\newcommand{\ie}{\emph{i.e. }}

\begin{document}
\author{Luke Oeding}
\address{Dipartimento di Matematica ``U. Dini'', Universita di Firenze, Viale Morgagni 67/A, 50134 Firenze (Italy)}
\email{oeding@math.unifi.it}
\date{\today}

\title[Set-theoretic defining equations]{Set-theoretic defining equations of the tangential variety of the Segre variety}



\thanks{This material is based upon work supported by the National Science Foundation under Award No. 0853000: International Research Fellowship Program (IRFP),
and U. S. Department of Education grant Award No. P200A060298: Graduate Fellowships for Ph.D. Students of Need in Mathematics (GAANN)}
\maketitle

\begin{abstract}
We prove a set-theoretic version of the Landsberg--Weyman Conjecture on the defining equations of the tangential variety of a Segre product of projective spaces.    
We introduce and study the concept of exclusive rank.
For the proof of this conjecture we use a connection to the author's previous work \cite{oeding_pm_paper, oeding_thesis} and re-express the tangential variety as the variety of principal minors of symmetric matrices that have exclusive rank no more than one.
\end{abstract}
\section{Introduction}

For each $i$, $1\leq i \leq n$, let $V_{i}$ be a complex vector space of dimension $n_{i}+1$ and let $V_{i}^{*}$ be the dual vector space.  The Segre product $Seg(\PP V_{1}^{*}\tdt \PP V_{n}^{*})$ is the variety of indecomposable tensors in $\PP(V_{1}^{*}\otdot V_{n}^{*})$. If $V$ is a complex vector space and $X \subset \PP V$ is any variety, the \emph{tangental variety} of $X$, denoted $\tau(X)$, is the union of all points on all embedded tangent lines  (\ie $\PP^{1}$'s) to $X$, \cite{Zak}.

Let $S_{\pi_{1}}V_{1}\otdot S_{\pi_{n}}V_{n}$ denote the irreducible $SL(V_{1})\tdt SL(V_{n})$ module associated to the partitions $\pi_{1},\dots,\pi_{n}$ of $d$.  Using and cohomological techniques and in particular Weyman's geometric method (see \cite{Weyman}), Landsberg and Weyman identified modules of this form in the ideal of $\tS$ and made the following conjecture.

\begin{conjecture}[Conjecture 7.6. \cite{LWtan}]\label{conj} $I(\tS)$ is generated by the quadrics in $S^{2} (V_{1}\otdot V_{n})$ 
which have at least four $\bigwedge^{2}$ factors, the cubics with four $S_{2,1}$ factors and all other factors $S_{3,0}$, 
and the quartics with three $S_{2,2}$'s and all other factors $S_{4,0}$. 
\end{conjecture}

The \emph{secant variety} of $X$, denoted $\sigma(X)$ is the variety of all embedded secant $\PP^{1}$'s to $X$, and since every tangent line to $X$ is the limit of secant lines, we have $\tau(X)\subset \sigma(X)$.

If $X = Seg(\PP V_{1}^{*}\tdt \PP V_{n}^{*})$, then $\sigma(Seg(\PP V_{1}^{*}\tdt \PP V_{n}^{*}))$ is contained in a subspace variety (or rank variety), namely $Sub_{2,\dots,2}(V_{1}^{*}\otdot  V_{n}^{*})$, which is all tensors $[T]\in \PP ( V_{1}^{*}\otdot  V_{n}^{*})$
such that there exists auxiliary subspaces  $V_{i}'^{*}$  with $\dim(V_{i}'^{*}) = 2$ for  $1\leq i\leq n$ and $T\in V_{1}'^{*}\otdot V_{n}'^{*}$. 


In \cite{LWtan} Landsberg and Weyman point out that because $\tau(X) \subset \sigma(X)$ and $\sigma(X)$ is in $Sub_{2,\dots,2}(V_{1}^{*}\otdot V_{n}^{*})$, it is sufficient to answer Conjecture \ref{conj} in the case that $V_{i}\simeq \CC^{2}$. We will prove the set-theoretic version of this conjecture in this case.

Our point of departure is to consider the (not immediately obvious) embedding of $\tS$ as a subvariety of $Z_{n}$ -- the variety of principal minors of symmetric $n\times n$ matrices.  We give a precise definition of $Z_{n}$ and list some of its properties in Section \ref{pmSection}.

In the case $n=3$ the ideal of the tangential variety $\tau(Seg(\PP^{1} \times \PP^{1} \times \PP^{1}))$ is defined by Cayley's hyperdeterminant of format $2\times 2 \times 2$, and this is the quartic equation in the conjectured ideal.

In \cite{HSt}, Holtz and Sturmfels showed that the ideal of $Z_{3}$ is generated by the same polynomial, therefore $\tau(Seg(\PP^{1}\times \PP^{1}\times \PP^{1})) = Z_{3}$.  In general the two varieties are not equal but one inclusion holds, namely $\tS \subset Z_{n}$ for all $n\geq 3$, \cite{oeding_thesis}.

Holtz and Sturmfels conjectured that the \emph{hyperdeterminantal module} -- the span of $\G$-orbit of Cayley's $2\times 2\times 2$ hyperdeterminant -- generates the ideal of $Z_{n}$. The hyperdeterminantal module is the module of quartic polynomials polynomials with three $S_{2,2}$'s and all other factors $S_{4,0}$, \ie the quartics in the Landsberg--Weyman Conjecture. In  \cite{oeding_pm_paper,oeding_thesis} we proved the set-theoretic version of the Holtz--Sturmfels Conjecture:

\begin{theorem}[\cite{oeding_pm_paper, oeding_thesis}]\label{pmtheorem}
Let $Z_{n} \subset \PP (V_{1}^{*} \otdot V_{n}^{*})$ be the variety of principal minors of symmetric matrices and let $HD$ be the module of quartic polynomials with three $S_{2,2}$'s and all other factors $S_{4,0}$.  Then, as sets, $\V(HD) = Z_{n}$.
\end{theorem}

In this paper we develop an understanding of the polynomials in the Landsberg--Weyman conjecture via their connection to $Z_{n}$. Using this connection we arrive at the following:
\begin{theorem}\label{main theorem} $\tS$ is cut out set-theoretically by the cubics in $S^{3} (V_{1}\otdot V_{n})$  with four $S_{2,1}$ factors and all other factors $S_{3,0}$, 
and the quartics in $S^{4} (V_{1}\otdot V_{n})$  with three $S_{2,2}$'s and all other factors $S_{4,0}$. 
\end{theorem}

\begin{remark}

Notice that we are not asking for any quadratic equations.
So we are proving something slightly stronger than the set-theoretic Landsberg--Weyman Conjecture. That is, we consider \emph{less} polynomials in the ideal, and show that these suffice to cut out the variety set-theoretically.
\end{remark}

Here is an outline of the rest of the paper and of the proof of Theorem \ref{main theorem}.
We know that all of the polynomials in the conjecture are in the ideal of the tangential variety. For the set-theoretic result, it remains to show that the tangential variety contains the zero set of these polynomials.

By Theorem \ref{pmtheorem} the set of quartics cuts out $Z_{n}$ so we need to show that the subvariety of $Z_{n}$ defined by intersecting with the zero set of the quadrics and cubics coincides with the tangential variety.

In Section \ref{polynomials} we explicitly construct the quadrics and cubics in the ideal of the tangential variety and then pull these polynomials back to the space of symmetric matrices.  We show that the quadric equations are unnecessary for the set-theoretic result. We then consider the subvariety $X \subset S^{2} \CC^{n}$ defined by this pullback.

The description of this variety $X$, motivates the introduction of the \emph{exclusive rank} (or \emph{E-rank}) of a matrix.  In Section \ref{E-rank} we define E-rank and in Proposition \ref{fixed} we show that E-rank is an invariant of $\G \subset GL(2n)$ with a natural action which we describe.

In Section \ref{E-rank01} we study the principal minors of E-rank zero and one symmetric matrices. Finally in Proposition \ref{E-rank one}, we show that the image under the principal minor map of the symmetric matrices with E-rank no larger than one is exactly the tangential variety.  This will show that every point in the zero set of the cubic and quartic polynomials in the conjecture has a symmetric E-rank one matrix in $X$ mapping to it under the principal minor map.  But the image of $X$ under the principal minor map is the tangential variety, the original point must be in the tangential variety and this completes the proof.

\section{The variety of principal minors of symmetric matrices}\label{pmSection}
To give a precise definition of $Z_{n}$ we need some notation.
Let $I = (i_1,  \dots  i_n)$ be a binary multi-index, with $i_k \in \{0,1\}$ for $k = 1,  \dots , n$, and let $| I | = \sum_{k=1}^{n}i_{k}$. For notational compactness, we will drop the commas and parentheses in the expression of $I$ when there is no danger of confusion.
 
If $A$ is an $n\times n$ matrix, then let $\Delta^{I}_{J}(A)$ denote the minor of $A$ formed by taking the determinant of the submatrix of $A$ with rows indexed by $I$ and columns indexed by $J$, in the sense that the submatrix of $A$ is formed by only including the $k^{th}$ row (respectively column) of $A$ whenever $i_{k} = 1$ (respectively $j_{k}=1$).  When $I=J$, the minor is said to be \emph{principal}, and we will denote it by $\Delta_{I} =\Delta_{I}^{I}$.

For $1\leq i \leq n$ let $V_{i}\simeq \CC^{2}$ and consider $V_1\otimes V_2\otimes \dots \otimes V_n \simeq \CC^{2^n}$. A choice of basis $\{x_i^0,x_i^1\}$ of $V_i$ for each $i$ determines a basis of $V_{1}\otimes \dots \otimes V_{n}$.  We represent basis elements compactly by setting $X^{I} := x_{1}^{i_{1}}\otimes x_{2}^{i_{2}}\otimes  \dots \otimes x_{n}^{i_{n}}$.  We use this basis to introduce coordinates on $\PP \CC^{2^{n}}$; if $P= [C_{I}X^{I}] \in \PP \CC^{2^{n}}$, the coefficients $C_{I}$ are the coordinates of the point $P$.

Let $S^{2}\CC^{n}$ denote the space of symmetric $n\times n$ matrices. The projective variety of principal minors of $n\times n$ symmetric matrices, $Z_n$, is defined by the following rational map,
\begin{eqnarray*}
 \varphi   : \PP(S^2 \CC^n \oplus  \CC)  & \dashrightarrow &  \PP \CC^{2^n} \\{}
 [A,t]  & \longmapsto & \left[t^{n-|I|}\Delta_{I}(A)\; X^I \right].
\end{eqnarray*}
The map $\varphi$ is defined on the open set where $t\neq0$. Moreover, $\varphi$ is homogeneous of degree $n$, so it is well defined on projective space.

\section{The pull-back of polynomials in the Landsberg--Weyman Conjecture symmetric matrices via $Z_{n}$}\label{polynomials}
\subsection{Background and notation}
If $X\subset \PP(V_{1^{*}}\otdot V_{n}^{*})$ is a variety invariant under the action of $SL(V_{1})\tdt SL(V_{n})$ -- of which $\tS$ and $Z_{n}$ are two examples -- we say $X$ is a $G$-variety for $G=SL(V_{1})\tdt SL(V_{n})$.  The ideal of such a $G$-variety is a $G$-submodule of $\bigoplus_{d}S^{d}(V_{1}\otdot V_{n})$.  Each degree-$d$ piece has an \emph{isotypic decomposition} which Landsberg and Manivel recorded in \cite{LM04} as follows:

\begin{prop}[Landsberg--Manivel \cite{LM04} Proposition 4.1]\label{LMdecomp}
 Let $V_1, \dots , V_n$ be vector spaces and let $V=V_1\otimes\dots\otimes V_{n}$, and let $G= GL(V_1)\tdt GL(V_n)$. Then the following decomposition as a direct sum of irreducible $G$-modules holds:
\[S^d(V_1\otdot V_n) = \bigoplus _{|\pi_1|=\dots=|\pi_k|=d} ([\pi_1]\otdot [\pi_k])^{\mathfrak{S}_d} \otimes S_{\pi_1}V_{1} \otdot S_{\pi_n} V_{n}
\]
where $[\pi_{i}]$ are representations of the symmetric group $\mathfrak{S}_{d}$ indexed by partitions $\pi_{i}$ of $d$, $([\pi_1]\otimes\dots\otimes[\pi_k])^{\mathfrak{S}_d}$ denotes the space of $\mathfrak{S}_d$-invariants (\ie, instances of the trivial representation) in the tensor product, and $S_{\pi_{i}}V_{i}$ are Schur modules.
\end{prop}
For more background on this decomposition formula see \cite{LM04} and for more background on representation theory one may consult \cite{FH}.
Note the representation theory for $SL(n)$ and $GL(n)$ is the same up to twists by determinants so we can also use this proposition when $G = SL(V_{1})\tdt SL(V_{n})$.  When $V_{i}$ are all isomorphic to the same $V$, we also have an $\mathfrak{S}_{n}$ action.  In this case, the irreducible $G$-modules for $G = SL(V)^{\times n} \ltimes \mathfrak{S}_{n}$ are direct sums of the modules of the form $S_{\pi_1}V_{1} \otdot S_{\pi_n} V_{n}$ where the $\pi_{i}$'s occur in every order that produces a non-redundant module.  In this case we often drop the superfluous notation of the tensor products and the vector spaces and denote the irreducible $SL(V)^{\times n} \ltimes \mathfrak{S}_{n}$-modules by $S_{\pi_1}\dots S_{\pi_n}$.

We construct polynomials from Schur modules via Young symmetrizers and Young tableau.
This construction is described in detail in \cite{LandsbergMorton} so we do not attempt to repeat its description here, but merely give a brief summary.

The basic idea is that for a partition $\pi$ of an integer $d$, each filled Young tableau of shape $\pi$  provides a recipe for constructing a certain Young symmetrizer, \ie a map $c_{\pi} :  V^{\otimes d} \rightarrow  V^{\otimes d}$ whose image is isomorphic to $S_{\pi}V \subset V^{\otimes d}$.  The map $c_{\pi}$ is defined by skew-symmetrizing over the columns and symmetrizing over the rows of the filled Young tableau of shape $\pi$.  In particular, one can construct a highest weight vector of $S_{\pi}V$ as the image under $c_{\pi}$ of a simple vector in $V^{\otimes d}$ of the correct weight. 

To construct a tensor in a module of degree $d$ polynomials of the from $S_{\pi_1}V_{1} \otdot S_{\pi_n} V_{n}$, we must make a clever combination of choices of fillings of the the Young tableau of shapes $\pi_{i}$ so that the resulting tensor in $S^{d}(V_{1}\otdot V_{n})$ is non-zero.  When $S_{\pi_1}V_{1} \otdot S_{\pi_n} V_{n}$ occurs with multiplicity $m>1$, we must repeat this process until we get $m$ linearly independent vectors to span the highest weight space (which by definition has dimension $m$).

\subsection{Constructing polynomials for the Landsberg--Weyman Conjecture}
We first consider the case $n=4$.  After this, we can build the polynomials for the general case from those in the base case.

Consider the $SL(2)^{\times 4}ltimes \mathfrak{S}_{4}$-module $\bigwedge^{2}\bigwedge^{2}\bigwedge^{2}\bigwedge^{2}$.  This module is one-di\-mensional and is the span of the polynomial
\begin{eqnarray*}
F_{0}=
X^{0000}X^{1111}
-X^{0001}X^{1110}
-X^{0010}X^{1101}
+X^{0011}X^{1100}
\\
-X^{0100}X^{1011}
+X^{0101}X^{1010}
+X^{0110}X^{1001}
-X^{0111}X^{1000}.
\end{eqnarray*}

We pull back $F_{0}$ to $S^{2} \CC^{n} \oplus \CC$ by making the substitution
$X^{I} = t^{n-|I|} \Delta_{I}(A)$,
where $\Delta_{I}(A)$ the principal minors (indexed by $I$) of a symmetric matrix $A = (a_{i,j})$. We find the polynomial
\[
F_{0}(A):=t^{4}\left(
\begin{array}{l}
a_{1,4}^{2}a_{2,3}^{2} + a_{1,3}^{2}a_{2,4}^{2} + a_{1,2}^{2}a_{3,4}^{2} 
\\
- a_{1,2}a_{2,3}a_{3,4}a_{1,4} -a_{1,2}a_{2,4}a_{1,3}a_{3,4} -a_{1,3}a_{2,4}a_{2,3}a_{1,4}
\end{array}\right).
\]
Notice that $F_{0}(A)$ is independent of the diagonal entries of $A$. 

Next, consider the module $S_{2,1}S_{2,1}S_{2,1}S_{2,1}$.  This module occurs with multiplicity $3$ in the decomposition of $S^{3}(V_{1}\otimes V_{2} \otimes V_{3} \otimes V_{4})$.  In order to get a basis of the highest weight space, we alter the fillings of the Young tableau in the standard construction of highest weight vectors via Young symmetrizers.  There are only 2 options for standard fillings in each of the 4 factors: $ \young(12,3)$, $ \young(13,2)$.  Of the possible $2^{4}$ constructions we must find $3$ which are linearly independent.  We found the following three basis vectors of the highest weight space via images of the Young symmetrizers defined by the fillings $T_{\pi_{1}},T_{\pi_{2}},T_{\pi_{3}},T_{\pi_{4}}$  via the recipes below:
\begin{itemize}
\item $T_{\pi_{1}}=T_{\pi_{2}}=T_{\pi_{3}}=T_{\pi_{4}}= \young(12,3)$ 
\begin{eqnarray*}
F_{1}&:=&\big(
X^{0000}X^{1111}
-X^{0001}X^{1110}
-X^{0010}X^{1101}
+X^{0011}X^{1100}
\\ &&
-X^{0100}X^{1011}
+X^{0101}X^{1010}
+X^{0110}X^{1001}
-X^{0111}X^{1000}
\big)2X^{0000}
\end{eqnarray*}

\item $T_{\pi_{1}}=T_{\pi_{2}}=  \young(12,3) ,\;\;\; T_{\pi_{3}}=T_{\pi_{4}}=  \young(13,2)$

\begin{eqnarray*}
F_{2}&:=&(2X^{0000}X^{1100}-2X^{0100}X^{1000})X^{0011} \\ &&
+(-X^{1100}X^{0001}+X^{0101}X^{1000}+X^{0100}X^{1001}-X^{1101}X^{0000})X^{0010}\\&&
+(-X^{0010}X^{1100}+X^{1000}X^{0110}+X^{1010}X^{0100}-X^{1110}X^{0000})X^{0001}\\&&
+(X^{0011}X^{1100}-X^{0111}X^{1000}-X^{0100}X^{1011}+X^{0000}X^{1111})X^{0000}
\end{eqnarray*}
\item $T_{\pi_{1}}=T_{\pi_{3}}= \young(12,3) , \;\;\;T_{\pi_{2}}=T_{\pi_{4}}= \young(13,2)$
\begin{eqnarray*}
F_{3} &:=&(2X^{0000}X^{0101}-2X^{0001}X^{0100})X^{1010}\\&&
+(-X^{0101}X^{0010}+X^{0100}X^{0011}+X^{0110}X^{0001}-X^{0111}X^{0000})X^{1000}\\&&
+(-X^{0101}X^{1000}+X^{0100}X^{1001}+X^{1100}X^{0001}-X^{1101}X^{0000})X^{0010}\\&&
+(X^{0101}X^{1010}-X^{0100}X^{1011}-X^{0001}X^{1110}+X^{0000}X^{1111})X^{0000}
\end{eqnarray*}
\end{itemize}

Notice that $F_{1} = 2X^{0000}F_{0}$.  This is an indication of the fact that the copy of $S_{2,1}S_{2,1}S_{2,1}S_{2,1}$ associated to the highest weight vector $F_{1}$ is in the ideal generated by $\bigwedge^{2}\bigwedge^{2}\bigwedge^{2}\bigwedge^{2}$. 
\subsection{The pullback of the cubic polynomials to $Z_{n}$}
We work on the open set $t\neq 0$,and set $t=1$.  On this set, $F_{0}$ and $F_{1}$ pull back to the same polynomial.  Since we are only working for the set-theoretic result, it suffices to just consider the $3$ copies of the module $S_{2,1}S_{2,1}S_{2,1}S_{2,1}$.

We find the following polynomials on the entries of the matrix $A$:
\begin{eqnarray*}
F1(A) &=&  
4(a_{1, 2}^2a_{3, 4}^2-a_{1, 2}a_{1, 3}a_{2, 4}a_{3, 4}-a_{1, 2}a_{1, 4}a_{2, 3}a_{3, 4}\\
&&+a_{1, 3}^2a_{2, 4}^2-a_{1, 3}a_{1, 4}a_{2, 3}a_{2, 4}+a_{1, 4}^2a_{2, 3}^2)
\\
F_{2}(A) &=&  
4a_{1, 2}^2a_{3, 4}^2-2a_{1, 2}a_{1, 3}a_{2, 4}a_{3, 4}-2a_{1, 2}a_{1, 4}a_{2, 3}a_{3, 4}\\
&&+a_{1, 3}^2a_{2, 4}^2-2a_{1, 3}a_{1, 4}a_{2, 3}a_{2, 4}+a_{1, 4}^2a_{2, 3}^2
\\
F_{3}(A) &=& 
a_{1, 2}^2a_{3, 4}^2-2a_{1, 2}a_{1, 3}a_{2, 4}a_{3, 4}-2a_{1, 2}a_{1, 4}a_{2, 3}a_{3, 4}\\
&&+4a_{1, 3}^2a_{2, 4}^2-2a_{1, 3}a_{1, 4}a_{2, 3}a_{2, 4}+a_{1, 4}^2a_{2, 3}^2
\end{eqnarray*}

We used Maple for the constructions of $F_{0},F_{1},F_{2}$, and $F_{3}$ above. Then we decomposed the ideal generated by $F_{1}(A),F_{2}(A),F_{3}(A)$ in Macaulay2
and got the single prime ideal
\begin{equation}\label{simple ideal}
\langle a_{1,3}a_{2,4}  - a_{1,4}a_{2,3},\;\;  a_{1,2}a_{3,4}- a_{1,4}a_{2,3} \rangle
.\end{equation}
We immediately recognize these equations as special $2\times 2$ minors of the symmetric matrix $A$.  In fact, these minors come from submatrices of $A$ which have no rows or columns in common.  We study such minors in Section \ref{E-rank}.

Next we used Maple to construct a basis of the three copies of the module $S_{2,1}S_{2,1}S_{2,1}S_{2,1}$ coming from the highest weight vectors $F_{1}, F_{2} $ and $F_{3}$. We pulled back these $48$ polynomials to the space of symmetric matrices. Then we decomposed this ideal in Macaulay2, and found that the same ideal as in (\ref{simple ideal}).  

We note that while the polynomials $F_{1}(A),F_{2}(A),F_{3}(A)$ do not depend on the diagonal terms of the matrix $A$, this does not hold for all of the other basis vectors in the module $S_{2,1}S_{2,1}S_{2,1}S_{2,1}$.  However, the radical of the ideal still does not depend on the diagonal terms of $A$.

In the general case, we consider the modules  $S_{2,1}S_{2,1}S_{2,1}S_{2,1}S^{3}\dots S^{3}$.  These modules have the same highest weight vectors (up to permutation) as those we considered in the above example, so the pullback of $S_{2,1}S_{2,1}S_{2,1}S_{2,1}S^{3}\dots S^{3}$ to symmetric matrices  must have (at least) all of the $2\times 2$ minors of the matrix $A$ which have no rows or columns in common in our ideal in the general case.

We have seen that the module 
$\bigwedge^{2}\bigwedge^{2}\bigwedge^{2}\bigwedge^{2}S^{2}\dots S^{2}$,
is not necessary for the set-theoretic question because it gives the same equations on the pull-back as one of the copies of $S_{2,1}S_{2,1}S_{2,1}S_{2,1}S^{3}\dots S^{3}$.  On the other hand, if one were to be interested in the minimal generators of the ideal of $\tS$ this module and all of the modules in degree $2$ would need to be considered.

\section{Exclusive rank}\label{E-rank}
In this section, motivated by the equations we found in the previous section, we study the minors whose row and column sets are disjoint.
We will say a minor $\Delta^{I}_{J}(A)$ is an \emph{exclusive minor} (or \emph{E-minor}) if $I\cap J= \emptyset$. The Laplace expansion expresses a $(k+2)\times (k+2)$ E-minor as a linear combination of $(k+1)\times (k+1)$ E-minors.  Therefore, if all the $(k+1)\times (k+1)$ E-minors vanish, then all the $(k+2)\times (k+2)$ E-minors vanish as well.  In light of this, we define the \emph{exclusive rank} (or \emph{E-rank}) of a matrix to be the minimal $k$ such that all the $(k+1)\times (k+1)$ E-minors vanish.

\begin{prop}\label{fixed}
The E-minors are fixed points under the action of $SL(2)^{\times n}$.  In particular, the E-rank is $G \simeq \G $ invariant.
\end{prop}
\begin{proof}

First note that it suffices to prove that a E-minor is taken to one of the same size under the action of $\G$.

The $\mathfrak{S}_n$ invariance is clear.  So we need to prove the first statement.  We recall the inherited action of $SL(2)^{\times n}$ as a subgroup of $GL(2n)$.  Seen in this way, we can give a proper definition of the action of $SL(2)^{\times n}$ on the exclusive minors.

Let $V = E \oplus F$ and let $E \simeq F \simeq \CC^{n}$.  The Grassmanian $G(n, V)$ can be parametrized by the rational map,
 \begin{eqnarray*} 
\psi: \PP  ( E^{*}\otimes F\oplus \CC) 
& \dashrightarrow & 
\PP \left(\bigwedge^{n} V\right)=\PP \left(\bigoplus_{k=0}^{n}\left(\bigwedge^{k}E^{*}\otimes\bigwedge^{k}F \right)\right)\\{}
[(A),t] & \longmapsto & \left[\sum_{|R|=|S|}t^{k-|R|}e^{R}\otimes f_{S}(A)\right].\end{eqnarray*}
The map $\psi$ is a variant of the Pl$\ddot{\text{u}}$cker embedding of the Grassmannian, and it is compatible with the decomposition of $\bigwedge^{n} V$. In light of this mapping $\psi$, the Grassmannian $Gr(n,2n)$ has the interpretation as the variety of (vectors of) minors of $n\times n$ matrices.   

For convenience, we will choose a volume form in $\bigwedge^{n}E$ and identify $\bigwedge^{n-k}E^{*}$ with $\bigwedge^{k}E$.  Then we will work with the minors as elements of $\bigwedge^{n-k}E \wedge \bigwedge^{k}F$ -- there is no harm in using a wedge between $E$ and $F$ because the vector spaces intersect only at the origin so we can interchange the tensor symbol with the wedge symbol -- and consider $e_{R} \wedge f_{S}(A) $ the minor of $A$ found by taking the determinant of the submatrix formed by keeping the rows of $A$ indexed by $R^{c}$ and the columns of $A$ indexed by $S$.   In this notation, the principal minors of $A$ are $e_{R} \wedge f_{S}(A)$ with $R \cap S = \emptyset$ and the E-minors of $A$ are $e_{R} \wedge f_{S}(A)$ with $R = S$.

Consider a vector $\sum_{|R|=|S|\geq 1}e_{R} \wedge f_{S}(A) $ of all minors of a given $n\times n$ matrix $A$.  Since this vector is in $G(n,2n)$, we can consider the action of $GL(2n)$ on it, and by the inclusion  $SL(2)^{\times n} \subset GL(2n)$ (given below) we can consider the action of $SL(2)^{\times n}$ on $e_{R} \wedge f_{S}(A) $.  

In \cite{HSt, oeding_thesis, BorodinRains}, it is shown that the action of $SL(2)^{\times n}$ preserves the variety of principal minors of symmetric matrices.  Here we will show that this action fixes the E-minors.

The inclusion we consider is the following.
\begin{align*}
SL(2)^{\times n} = \left\{  \left(\begin{array}{cc} D_{1} & D_{2} \\ D_{3} & D_{4} \end{array} \right) \mid D_{i} -\text{diagonal}, D_{1}D_{3}- D_{2}D_{4} = I_{n} \right\} 
\\
\subset  \left\{ M \in  \left( \begin{array}{cc} E^{*}\otimes E  & F^{*} \otimes E \\ E^{*} \otimes F & F^{*}\otimes F  \end{array}  \right) \mid det(M) \neq 0 \right \} = GL(V).
\end{align*}
Consider the blocked matrix $g =\left(\begin{array}{cc} a^{i}_{i} & b^{j}_{j} \\ c^{k}_{k} & d^{l}_{l} \end{array} \right) \in SL(2)^{\times n}$ with $1\leq i,j,k,l \leq n$. The individual elements of each $SL(2)$ are the $2\times 2$ matrices constructed from $g$ as $\left(\begin{array}{cc} a^{i}_{i} & b^{i}_{i} \\ c^{i}_{i} & d^{i}_{i} \end{array} \right)$.
For simplicity, let all factors of $g$ except the first factor be the identity matrix and consider the action on a exclusive minor
\[
g. e_{R}\otimes f_{R} = g.\big(e_{i_{1}}\wedge e_{i_{2}} \wedge \dots \wedge e_{i_{|R|}} \wedge f_{i_{1}}\wedge \dots \wedge f_{i_{|R|}} \big)
\]
\[
= \big((a_{i_{1}}^{i_{1}}e_{i_{1}} + c^{i_{1}}_{i_{1}} f_{i_{1}} )\wedge e_{i_{2}} \wedge \dots \wedge e_{i_{|R|}} \wedge ( b^{i_{1}}_{i_{1}} e_{i_{1}} + d^{i_{1}}_{i_{1}}f_{i_{1}} )\wedge f_{i_{1}}\wedge \dots \wedge f_{j_{|S|}} \big)
.\]

But if we expand this expression, and use the fact that $e_{i_{1}}\wedge e_{i_{1}} = f_{i_{1}}\wedge f_{i_{1}} = 0$ we see that the only nonzero term is 

\[
= (a_{i_{1}}^{i_{1}} c_{i_{1}}^{i_{1}} - b_{i_{1}}^{i_{1}} d_{i_{1}}^{i_{1}}) e_{R}\wedge f_{R} = e_{R}\wedge f_{R} 
.\]
Therefore the exclusive minors are fixed by $SL(2)^{\times n}$. 
%
\end{proof}

\section{Principal minors of symmetric matrices with small exclusive rank}\label{E-rank01}

In this section we study the symmetric matrices that have E-rank less than or equal to one and their principal minors.  The main goal of this section is Proposition \ref{E-rank one}, which is the key to the proof of Theorem \ref{main theorem}.  First we consider the case of E-rank zero symmetric matrices.  To prove Proposition \ref{E-rank one} we first consider the principal minors of honest rank one symmetric matrices in Proposition \ref{veronese}.  We show that the $\G$ orbit of rank one symmetric matrices is $\tS$.  Then we show that the variety of principal minors of E-rank one symmetric matrices is the tangential variety by showing that it is $\G$-invariant, irreducible, and has the same dimension as the orbit of principal minors of the the honest rank one symmetric matrices.  The $\G$-invariance comes from Lemma \ref{equivariant} which is a general statement about how symmetry can be preserved under a projection from a $G$-variety.

\begin{prop}\label{diagonal}\cite{oeding_thesis}.
Let $U_{0}= \{[z_{I}X^{I}] \in \PP(V_{1}^{*}\otdot V_{n}^{*}) \mid z_{[0,\dots,0]} \neq 0 \}$.  $\varphi([A,t])\in Seg(\PP V_{1}^{*}\tdt \PP V_{n}^{*} ) \cap U_{0}$, if and only if $A$ is diagonal (has E-rank $0$).
\end{prop}

\begin{proof}
Let $\{x_{i}^{0},x_{i}^{1}\}$ be a basis of $V_{i}^{*}$ for each $i$.  Let $z = (a^{1}x_{1}^{0}+ b^{1}x_{1}^{1})\otdot (a^{n}x_{n}^{0}+ b^{n}x_{n}^{1}) $ be such that $[z]\in Seg(\PP V_{1}^{*}\tdt \PP V_{n}^{*} ) \cap U_{0}$ and suppose $A$ is a matrix such that $\varphi([A,t]) = [z]$.
The following relations on the $0\times 0$ and $1\times 1$ principal minors of $A = (x_{i,j})$ must hold:  \[t^{n}= (a^{1}\dots a^{n}) = z_{[0,\dots,0]}\]
\[ t^{n-1}x_{i,i} = (a^{1}\dots a^{i-1} b^{i} a^{i+1}\dots a^{n}) = z_{[0,\dots,0,1,0,\dots,0]}.\]
We are assuming that $z_{[0,\dots,0]}\neq 0$, so this implies that $a^{i} \neq 0 \;\forall i$ and that $t\neq 0$, so we can solve these equations to find $x_{i,i} = \frac{b^{i}}{a^{i}} t$.
Also, the following relation on $2\times 2$ minors must hold,
\[ t^{n-2}(x_{i,i}x_{j,j} - x_{i,j}^{2}) = (a^{1}\dots a^{i-1} b^{i} a^{i+1}\dots a^{j-1}b^{j} a^{j+1}\dots a^{n} ),\]
which implies that $x_{i,j}=0$ for all $i\neq j$.  Therefore $A$ must be a diagonal matrix.

For the converse, suppose $A = (x_{i,j})$ is diagonal, we must show that $\varphi([A,1]) \in Seg(\PP V_{1}^{*} \tdt \PP V_{n}^{*})$.  We can further suppose that $x_{i,i}$ are of the form $x_{i,i} = \frac{b^{i}}{a^{i}} t$ for some constants $b^{i}$, $a^{i}$ and $t$ with the $a^{i}$, assumed to be nonzero and $t^{n} = a^{1}\dots a^{n}$.
Because $A$ is assumed diagonal, its principal minors are easy to calculate: Let $I(p)$ is a multi-index with $1$'s in the positions $p_{1},\dots,p_{k}$ and $0$'s elsewhere, then 
\begin{align*}
t^{n-k}\Delta_{I(p)}(A) = t^{n-k} x_{p_{1},p_{1}}\dots x_{p_{k},p_{k}} \\
=  (\frac{b^{p_{1}}}{a^{p_{1}}} t) \dots (\frac{b^{p_{k}}}{a^{p_{k}}} t) =  (\frac{b^{p_{1}} \dots b^{p_{k}}}{a^{p_{1}}\dots a^{p_{k}}} t^{n})\\
=  (\frac{b^{p_{1}} \dots b^{p_{k}}}{a^{p_{1}}\dots a^{p_{k}}} a^{1}\dots a^{n}).
\end{align*}
But the term $ (\frac{b^{p_{1}} \dots b^{p_{k}}}{a^{p_{1}}\dots a^{p_{k}}} a^{1}\dots a^{n})$ is the $I(p)$ coefficient of the expansion of the tensor, $z = (a^{1}x_{1}^{0}+ b^{1}x_{1}^{1})\otdot (a^{n}x_{n}^{0}+ b^{n}x_{n}^{1}) $, so we have $t^{n-k}\Delta_{I(p)}(A) = z_{I(p)}$ and $[z] \in Seg(\PP V_{1}^{*}\tdt \PP V_{n}^{*} ) \cap U_{0}$ as required. 
\end{proof}

\begin{prop}\label{E-rank one}
The tangential variety is the image of the E-rank one symmetric matrices under the principal minor map.
\end{prop}

To prove Proposition \ref{E-rank one}, we will use Proposition \ref{veronese} and Lemma \ref{glemma} below.
\begin{remark}
Though it is not necessary for this paper, it would be interesting to have a similar geometric description of the principal minors of the E-rank $k$ symmetric matrices for all $k$.  This would provide a geometric stratification of $Z_{n}$ by E-rank and would enhance our understanding of the geometry of $Z_{n}$.
\end{remark}

Consider the Veronese embedding of $\CC^{n}$ into the $n\times n$ matrices.
\begin{eqnarray*}
v_{2}: \CC^n & \longrightarrow &  S^{2} \CC^{n}\\
(y_{1},y_{2}, \dots ,y_{n}) & \longmapsto & 
\left(\begin{array}{cccc}
y_{1}^{2} & y_{2}y_{1} &  \dots  & y_{n}y_{1}\\
y_{1}y_{2} & y_{2}^{2} &  \dots  & y_{n}y_{2}\\
\vdots & \vdots & \ddots & \vdots\\
y_{1}y_{n} & y_{2}y_{n} &  \dots  & y_{n}^2\end{array}\right)
=\mathbf{y}.^{t}\mathbf{y}.\end{eqnarray*}
This parameterizes the rank one complex symmetric $n\times n$ matrices.
\begin{prop}\label{veronese}
 The $G$-orbit of the image (under $\varphi$) of the rank one symmetric matrices is the tangential variety to the $n$-factor Segre variety.  In particular, $\tS\subset Z_{n}.$
\end{prop}
\begin{proof}
For this proof only, let $Y:=\varphi(\PP (v_{2}(\CC^{n}))\oplus \CC)$.
We want to show that $G.Y = \tS$.

Since $\mathbf{y}.^{t}\mathbf{y}$ is a rank one symmetric matrix, all $k\times k$ minors vanish for $k>1$, and in particular, the $k\times k$ principal minors vanish for $k>1$.  Therefore a generic point in $Y$ has the form
\[
P=\left[ t\left(x_{1}^{0}\otdot x_{n}^{0}\right) + \sum_{i=1}^{n}y_{i}^{2}\left(x_{1}^{0}\otdot x_{i-1}^{0}\otimes x_{i}^{1}\otimes x_{i+1}^{0} \dots \otimes x_{n}^{0}\right)  \right]
,\]
where  $y_{i},t\in\CC$.  Consider the a curve 
\[\gamma(s) = x_{1}(s) \otdot x_{n}(s),\;\; s\in \CC
\]
such that $x_{i}(0) = x_{i}^{0}$ and the derivatives $x_{i}'(0) = x_{i}^{1}$. Then it is clear that $\gamma$ is a curve in $Seg(\PP V_{1}^{*}\tdt \PP V_{n}^{*})$ through $x_{1}^{0}\otdot x_{n}^{0}$, and that $P$ is on the tangent line to $\gamma$ at $s=0$.  So $P \in \tS$ and therefore  $Y \subset \tS$, and $\G.Y\subset \tS$ by the $\G$-invariance of $\tS$.

In the other direction, suppose we are given an arbitrary point 
\[\begin{array}{r}Q= \big[r_{0} (q_{1}\otdot q_{n})+\sum_{i}r_{i}(q_{1}\otdot q_{i-1}\otimes q_{i}'\otimes q_{i+1}\otdot q_{n}) \big]
\\ \in \tS
\end{array},\]  
where $r_{i}\in \CC$ for $0\leq i \leq n$ not all zero, and (without loss of generality), each pair $q_{i}, q_{i}'$ is a linearly independent pair so that $ \{q_{i}, q_{i}'\} =  V_{i}^{*}$. The form of $Q$ is generic up to the action of $\G$ so by changing basis on each $V_{i}$ by an $SL(2)$ action, we can assume $x_{i}^{0}=q_{i}$ and $x_{i}^{1}=q_{i}'$ for each $i$.






So $Q$ is in the $\G$-orbit of a point of the form
\[
P=\left[ r_{0}(x_{1}^{0}\otdot x_{n}^{0}) + \sum_{i=1}^{n} r_{i}\left(x_{1}^{0}\otdot x_{i-1}^{0}\otimes x_{i}^{1}\otimes x_{i+1}^{0} \dots \otimes x_{n}^{0}\right)  \right]
,\]
which is the image under $\varphi$ of the point $[\mathbf{y}.^{t}\mathbf{y},t]$,
where $t$ and $y_{i}$ are chosen such that $t^{n}=r_{0}$ and  $r_{i} = y_{i}^{2}t^{n-1}$.
This implies that $$\tS \subset \G.Y.$$  Therefore $\G.\varphi(v_2(\PP^n)) = \tS $. 
\end{proof}


The following lemma will be used in the proof of Lemma \ref{glemma} below.
\begin{lemma}\label{equivariant}
Let $T$ be a $G$-module and let $X\subset \PP T$ be a $G$-variety.  Let $H<G$ be a subgroup which splits $T$ - \emph{i.e.} $T=W\oplus W^{c}$ as an $H$-module.  Let $\pi : \PP(W \oplus W^{c}) \dashrightarrow \PP(\left(W\oplus W^{c}\right) / W^{c}) \simeq \PP W$ be the projection map.   The map $\pi$ is obviously $H$-equivariant, so the image $\pi(X)$ is an $H$-invariant subvariety of $\PP W$. 
\end{lemma}
This lemma tells us that if we are presented with a variety that is the projection from a $G$-variety, then we should look for the symmetry group of our variety among subgroups of $G$.
\begin{proof}
We must consider the fact that $\pi$ is only a rational map: certainly, $\pi (x) = 0$ if $x\in W^{c}$, so the map is not defined at all points.  

Let $U$ be the open set defined by $U = \{[w_{1}+w_{2}] \mid w_{1}\neq 0, w_{1}\in W, w_{2}\in W^{c}\}$.  Let $U_{X} = U \cap X$ denote the relatively open set.  

Let $Y := \overline{\pi(U_{X})}$, where the bar denotes Zariski closure.  
Claim: $H.U_{X} \subset U_{X}$.  Suppose $h \in H$ and $[w_{1}+w_{2}] \in U$. Then $h.[w_{1}+w_{2}] = [h.w_{1}+ h.w_{2}] \in U$ since $0\neq h.w_{1}\in W$ and $h.w_{2}\in W^{c}$.
Since $X$ is preserved by $G$, it is also preserved by any subgroup $H<G$, and therefore we conclude that $H.U_{X} \subset U_{X}$.

Let $y\in \pi(U_{X})$ and let $h\in H$.  By definition, $\pi$ is surjective onto its image, so let $x\in U_{X}$ be such that $\pi(x)=y$.  Now we use the $H$-equivariance of $\pi$ to conclude that $h.y = h.\pi(x) = \pi(h^{-1}.x)$.  But by the claim, we know that $h^{-1}.x \in U_{X}$, so $\pi(h^{-1}.x) \in \pi(U_{X})$.

Suppose $y\in \overline{\pi(U_{X})}$.  Then choose a sequence $y_{i} \rightarrow y \in Y$ such that $\exists x_{i}\in U_{X}$ and $\pi(x_{i})=y_{i}$.

If $h\in H$ then $h.y_{i}=h.\pi(x_{i})=\pi(h^{-1}.x_{i})\in Y$ for all $i$.  If $\{p_{i}\}\subset Y$ is a convergent sequence such that $p_{i}\rightarrow p$, and $f$ is a polynomial which satisfies $f(p_{i}) =0 $, then by continuity, $f(p) =0$ also. So $Y$ must contain all of its limit points, and therefore $h.y_{i}\rightarrow h.y\in Y$, and  we conclude that $Y$ is an $H$-variety.
\end{proof}


\begin{lemma}\label{glemma}
Let $X$ be the variety of $n\times n$ symmetric matrices which have E-rank one or less.
Then $X$ is an irreducible variety of dimension $2n$, and moreover the image $\varphi(X)$ is an irreducible $G$-variety for $G \simeq \G \subset Sp(2n)$.

\end{lemma}


%


\begin{proof}
\textbf{Claim 1:}  $\varphi(X)$ is an irreducible $\G$-variety. 
The map $\varphi$ is a rational map, so the fact that $\varphi(X)$ is an irreducible variety will come from the next claim that $X$ is irreducible.  Here we prove the $\G$-invariance. Our proof is similar to methods used in \cite{oeding_thesis} in the study of the symmetry of $Z_{n}$.

Let $\Gamma_{n} \simeq \CC^{\binom{2n}{n}-\binom{2n}{n-2}}$ denote the space of all non-redundant minors of $n\times n$ symmetric matrices.
Let $G_{\omega}(n,2n)\subset \PP \Gamma_{n}$ denote the Lagrangian Grassmannian embedded by the a variant of the map $\psi$ which we introduced in the proof of Proposition \ref{fixed} that takes a symmetric matrix to a vector of its non-redundant minors.  This map and its variants were studied in a more general context by Landsberg and Manivel \cite{LM02}, and the fact that this variant of $\psi$ defines $G_{\omega}(n,2n)$ can be found in \cite{LM02}. $G_{\omega}(n,2n)$ is a homogeneous variety and in particular it is invariant under the action of the symplectic group $SP(2n)$.

Let $\pi:G_{\omega}(n,2n) \dashrightarrow Z_{n}$ denote the projection by forgetting the non-principal minors. We will use Lemma \ref{equivariant} to prove the $\G$-invariance of $\varphi(X)$ by checking that the hypotheses are satisfied. 

Consider the linear space $L\subset \PP \Gamma_{n}$ defined by setting all $k\times k$ E-minors for $k\geq 2$ equal to zero.  Then by definition $\pi(G_{\omega}(n,2n)\cap \PP L) = \varphi(X)$.
Proposition \ref{fixed} implies that $L$ is fixed by the action of $\G$. 

$\Gamma_{n}$ is an $SP(2n)$-module, so by restriction, it is also an $\G$-module.  We further note that $L$ is a vector sub-space of $\Gamma_{n}$ and a $\G$-module, so it is a $\G$-submodule of $\Gamma_{n}$.

One can check that the inclusion of $SL(2)^{\times n}$ as a subgroup of $GL(2n)$ we gave in the proof of Proposition \ref{fixed} actually is an inclusion of $SL(2)^{\times n}$ as a subgroup of $SP(2n)\subset GL(2n)$.  
So $\G$ must act on $G_{\omega}(n,2n)$ and leave it invariant, and in particular $ G_{\omega}(n,2n)\cap \PP L$ is $\G$-invariant.  So we have satisfied the hypotheses of Lemma \ref{equivariant}, with $G = SP(2n)$, $H=\G$, $T = \Gamma_{n}$, and $ W=L$ ($W^{c}$ exists because $\G$ is reductive). So Lemma \ref{equivariant} implies that the image of $G_{\omega}(n,2n)\cap L$ under the projection $\pi$ is a $\G$-variety.

\textbf{Claim 2:} $X$ is irreducible. We work on the set where $t\neq 0$. Consider the following set of matrices, $Y= \{ A\in  S^{2}V  \mid A = D + T, D\text{ diagonal},  T \in v_{2}(\PP V) \}$. 

The variety $Y$ can be parametrized via
\[
\hspace{-12em}\PP (\CC^{n} \oplus \CC^{n} \oplus \CC) \dashrightarrow \PP(S^{2}\CC^{n} \oplus \CC)
\]  
\[
[w_{1}, \ldots, w_{n}, y_{1},\ldots,y_{n}, t] \mapsto 
\left[ \left(\begin{array}{cccccc}
w_{1}^{2} &  y_{1}y_{2} & \dots &\dots & y_{1}y_{n} \\
 y_{1}y_{2} &w_{2}^{2} & y_{2}y_{3} &\dots & y_{2}y_{n} \\
\vdots & y_{2}y_{3}& w_{3}^{2} &\vdots & \vdots \\

\vdots & \vdots & \dots& \ddots &y_{n-1}y_{n}\\

 y_{1}y_{n} &y_{2}y_{n} & \dots & y_{n-1}y_{n} & w_{n}^{2} \\
\end{array}
\right), t^{2} \right]
.\]
Then it is clear that all of the $2\times 2$ E-minors vanish on $Y$, so $Y \subset X$.

For the general case, we need to see that every matrix which has E-rank one can be expressed in this form.  Work by induction.  The base case is trivial.
Now suppose \[
A=
\left(\begin{array}{ccccccc}
w_{1}^{2} &  y_{1}y_{2} & \dots &\dots & y_{1}y_{n} & a_{1,n+1}\\
 y_{1}y_{2} &w_{2}^{2} & y_{2}y_{3} &\dots & y_{2}y_{n} &a_{2,n+1} \\
\vdots & y_{2}y_{3}& w_{3}^{2} &\vdots & \vdots  & \vdots \\

\vdots & \vdots & \dots& \ddots &y_{n-1}y_{n} &a_{n-1,n+1}\\

 y_{1}y_{n} &y_{2}y_{n} & \dots & y_{n-1}y_{n} & w_{n}^{2}  & a_{n,n+1}\\

a_{1,n+1} &a_{2,n+1} & \dots & a_{n-1,n+1} & a_{n,n+1}  & a_{n+1,n+1}
\end{array}
\right)
,\]
where we have assumed by induction that the upper left block of $A$ is in the desired form.

The $2\times 2$ E-minors force the vectors $( y_{1}y_{n} , y_{2}y_{n} , \dots , y_{n-1}y_{n})$ and \linebreak $(a_{1,n+1} , a_{2,n+1} , \dots , a_{n-1,n+1})$ to be proportional, so without loss of generality we may assume that $a_{i,n+1} = y_{i}y_{n+1}$ for $ 1\leq i\leq n-1$, and $y_{n+1}$ an arbitrary parameter.  By comparing to the first column, we find that the vectors $(y_{1}y_{2}, \dots , y_{1}y_{n})$ and $(a_{2,n+1} , a_{2,n+1} , \dots , a_{n,n+1})$ must be proportional, and therefore $a_{i,n+1} = y_{i}y_{n+1}'$ for $ 2\leq i\leq n$, and $y_{n+1}'$ an arbitrary parameter.  Combining this information, we must have $a_{j} = y_{j}y_{n+1} = y_{j}y_{n+1}'$ for $2\leq j\leq n-1$. If $y_{j} \neq 0$ for a single $j$ with $2\leq j\leq n-1$ then we find that $y_{n+1}= y_{n+1}'$, and in this case, $A$ is in the desired form.  Otherwise,

\[
A = 
\left(\begin{array}{ccccccc}
w_{1}^{2} & 0 & \dots &\dots & 0 & a_{1,n+1}\\
0 &w_{2}^{2} &0 &\dots & 0 &a_{2,n+1} \\
\vdots & 0 & w_{3}^{2} &\vdots & \vdots  & \vdots \\

\vdots & \vdots & \dots& \ddots & 0  &a_{n-1,n+1}\\

0 & 0  & \dots & 0 & w_{n}^{2}  & a_{n,n+1}\\

a_{1,n+1} &a_{2,n+1} & \dots & a_{n-1,n+1} & a_{n,n+1}  & a_{n+1,n+1}
\end{array}
\right).\]

But this is also in the form we want because (over $\CC$), we can set $a_{i,n+1} = y''_{i} y''_{n+1}$ for $1\leq i \leq n$ and $a_{n+1,n+1} = w_{n+1}^{2}$ for some arbitrary parameters $ y''_{i}$, $y''_{n+1}$, and $w_{n+1}$.

\textbf{Claim 3:} $X$ has dimension $2n$.  This is clear from the parameterization in the previous claim.  The map we gave is generically finite to one and the source has dimension $2n$.
\end{proof}

\begin{proof}[Proof of Proposition \ref{E-rank one}]
By Proposition \ref{veronese} above, 
\[G.\varphi(\PP(v_{2}(\CC ^{n}) \oplus \CC)) = \tS,\] 
where $v_{2}(\CC^{n})$ is the rank one complex symmetric $n\times n$ matrices.
By Lemma \ref{glemma} we have $\varphi(X) = G.\varphi(X)$,
 and since the condition rank one is more restrictive than the condition $E$-rank one, $X \supset \PP(v_{2}(\CC ^{n}) \oplus \CC) $, therefore 
\[\varphi(X) = G. \varphi(X) \supset G. \varphi(v_{2}(\PP V) \oplus \CC) = \tau(Seg(\PP^{1}\times\dots\times \PP^{1})).\]
So we have $\varphi(X)   \supset  \tau(Seg(\PP^{1}\times\dots\times \PP^{1}))$, an inclusion of two varieties that are both irreducible and of the same dimension, therefore we must have equality.
\end{proof}

\section{Conclusion}

In summary, to prove the set-theoretic version of the Landsberg-Weyman conjecture, we needed to show that the tangential variety 
$\tS$ contains the zeroset of the polynomials coming from the modules of cubic polynomials with four $S_{2,1}$ factors and the rest $S^{3}$ and the quartic polynomials with three $S_{2,2}$ factors and the rest $S^{4}$.  The quartic polynomials are set-theoretic defining equations of the variety of principal minors of symmetric matrices.  We studied the pull-back of the cubic polynomials to the space of symmetric matrices and found that this pull-back defines the set of E-rank one symmetric matrices.  Finally we showed that the image of the set of E-rank one symmetric matrices under the principal minor map is precisely the tangential variety.  Therefore if $z$ in the zeroset of the cubic and quartic polynomials in our modules, then $z$ has a E-rank one symmetric matrix $A$ mapping to it under the principal minor map, thus implying that $z$ is on the tangential variety. This completes the proof of Theorem \ref{main theorem}.

\section*{Acknowledgments }
This work stemmed from conversations with J.M. Landsberg, the author's thesis advisor.  It was he who suggested that the author try to find a meaningful notion of rank for the variety of principal minors, and he also pointed out that the hyperdeterminantal module of the Holtz-Sturmfels Conjecture also appeared in the Landsberg-Weyman Conjecture. We thank him for his advice and direction.

We also thank Prof.'s Rick Miranda and Ciro Ciliberto for suggesting the proof of Claim 2 in Lemma \ref{glemma}.

\bibliographystyle{amsplain}
\bibliography{bibdata}

\end{document}